\documentclass[reqno, 12pt]{amsart}

\usepackage{amsfonts, amsthm, amsmath, amssymb}
\usepackage{hyperref}
\hypersetup{colorlinks=false}

\usepackage[margin=1.5in]{geometry}

\usepackage{helvet}

\RequirePackage{mathrsfs} \let\mathcal\mathscr

\usepackage{mathabx,yfonts}
\usepackage{enumerate}
\usepackage{bbm}
\usepackage{colonequals}
\usepackage{microtype}
\usepackage{fullpage}
\usepackage{hyperref}
\usepackage{dsfont}
\usepackage{upgreek}
\numberwithin{equation}{section}

\newtheorem{theorem}{Theorem}[section] 
\newtheorem{lemma}[theorem]{Lemma}
\newtheorem{proposition}[theorem]{Proposition}
\newtheorem{corollary}[theorem]{Corollary}

\theoremstyle{definition}
\newtheorem*{acknowledgements}{Acknowledgements}
\newtheorem{remark}[theorem]{Remark}
\newtheorem{definition}[theorem]{Definition}

\newtheorem{notation}{Notation}

\newcommand{\e}{\mathrm{e}}

\newcommand{\li}{\mathrm{li}}
\renewcommand{\l}{\left}
\renewcommand{\r}{\right}

\title[Counting primes with a given primitive root in a progression]{Uniformly counting primes with a given primitive root and in an arithmetic progression}

\author{M. Zoeteman}
\email{m.d.zoeteman@umail.leidenuniv.nl}
\date{\today}

\begin{document}

\begin{abstract}
We study the number of primes with a given primitive root and in an arithmetic progression under the assumption of a suitable form of the generalized Riemann Hypothesis. Previous work of Lenstra, Moree and Stevenhagen has given asymptotics without an explicit error term, we provide an explicit error term by combining their work with the method of Hooley regarding Artin's primitive root conjecture. We give an application to a Diophantine problem involving primes with a given primitive root.
\end{abstract}

\maketitle

\setcounter{tocdepth}{1}
 
 \section{Introduction}
 The Artin conjecture on primitive roots states that for $g \in \mathbb{Z}$ not a perfect square and not equal to $-1$, there are infinitely many primes $p$ such that $g$ is a primitive root modulo $p$. There is no $g$ for which the conjecture is currently known. In 1967, Hooley \cite{Hooleyarticle} proved Artin's conjecture under the assumption of a suitable form of the generalised Riemann Hypothesis and gave an asymptotic for the number of such primes $p \leq x$. In 1977, Lenstra \cite{lenstra77} proved, also under the assumption of some form of the generalized Riemann Hypothesis, that the primes $p \equiv a \mod m$ with $\langle g \rangle = \mathbb{F}_p^*$ have a natural density, denoted by $\delta(a,m,g)$. Later, more work was done by Lenstra, Moree and Stevenhagen (\cite{lenstrastevenhagenmoree}, \cite{moree2008}) to study its expression. In our main result, the existence of the natural density is promoted to an asymptotic with an explicit error term. Let $\pi_g(x;m,a)$ be the number of primes $p \leq x$ with primitive root $g$ that satisfy $p \equiv a \mod m$. We write GRH$(L,g)$ for the hypothesis that for each squarefree $k \in \mathbb{N}$ all Hecke-$L$-functions of the number field $\mathbb{Q}(\zeta_k, \sqrt[k]{g})$ satisfy the Riemann Hypothesis, where $\zeta_k \in \mathbb{C}$ denotes a $k$-th primitive root of unity. In our main result we prove under GRH$(L,g)$ the following asymptotic.
 \begin{theorem} \label{thm:a.p.'s in a.p.'ssecondversion}
Let $g$ be an integer not equal to a perfect square and not equal to $-1$, and assume GRH$(L,g)$. For all $a,m \in \mathbb{N}$ with $\gcd (a,m)=1$, we have
 $$\pi_g\l(x;m,a\r)= \delta\l(a,m,g\r)  \frac{x}{\log x} +O \l(\frac{1}{\phi(m)} \frac{x }{\log x} \frac{ (\log |g|)  \max \l\{ \log (2m), \log \log x  \r\}}{\log x} \r),$$
 where the implied constant in the error term is absolute and where $\phi$ denotes Euler's totient function.
 \end{theorem}
 \begin{remark} After Theorem \ref{thm:moree2008}, we show that if the density $\delta(a,m,g)$ is positive, then $\delta(a,m,g) \gg \frac{1}{\phi(m) \log \log m}$, so the error term is dominated by the main term even when $m \leq (\log x)^{\sqrt{\log x}}$.
 \end{remark}
 This is the first time in the literature that the error term in Lenstra's result is made explicit. Let us now compare our result with the classical Siegel--Walfisz theorem for primes in an arithmetic progression. 
 If $m \leq (\log x)^N$, for some fixed $N>0$, then Theorem \ref{thm:a.p.'s in a.p.'ssecondversion} gives an asymptotic with worse error term with respect to $x$ than the Siegel--Walfisz theorem. In the range where
 $(\log x)^{f(x)} < m < x^{F(x)}$, for some functions $f(x)$ and $F(x)$ such that $f(x) \to \infty$ and $F(x)  = O \l( \frac{1}{\log \log x} \r)$ as $x \to \infty$, our result gives an asymptotic, whereas the Siegel--Walfisz result for primes without a primitive root restriction is not applicable. \\\\
The uniformity of the error term on $m$ is important in applications, since it means that the error term gets smaller when the modulus of the progression increases.
In the proof of our result, we will adapt Hooley's method \cite{Hooleyarticle} regarding the Artin primitive root conjecture, by modifying his functions $\xi _i(x)$ to depend also on $m$. Furthermore, we shall draw upon the works of Lenstra, Moree and Stevenhagen on the density $\delta(a,m,g)$.

\subsection{Overview of earlier results}
 We now give the precise statements of the results proved by Hooley, Lenstra, Moree and Stevenhagen mentioned in this introduction.
  \begin{notation}
The letters $p$ and $q$ will always denote a prime. Let $\mathcal{G}$ be the set of integers not equal to $-1$ or a perfect square. The notation $\l( \frac{\cdot}{\cdot} \r)$ is used for the Kronecker symbol. For integers $a$ and $m$, we write $(a,m)$ for $\gcd(a,m)$, and $[a,m]$ for lcm$(a,m)$.  For $k \in\mathbb{N}$, we denote by $\zeta_k$ a primitive $k$-th root of unity.
\end{notation}
 
 For $g \in \mathcal{G}$, write GRH$(g)$ for the hypothesis that for each squarefree $k \in \mathbb{N}$, the Dedekind zeta function of the number field $\mathbb{Q}(\zeta_k, \sqrt[k]{g})$ satisfies the Riemann hypothesis.

\begin{notation}
For $g \in \mathcal{G}$, let $\pi_g(x)$ denote the number of primes $p \leq x$ with primitive root $g$. If $a$ and $m$ are positive integers, let
$\pi_g(x;m,a)$ be the number of primes $p \leq x$ with primitive root $g$ that satisfy $p \equiv a \mod m$.
\end{notation}
In 1967, Hooley \cite{Hooleyarticle} proved under assumption of GRH$(g)$ that there are infinitely many primes with a prescribed primitive root $g$.
In fact, he proved the following asymptotic, where we write $g=g_1g_2^2$ with $g_1$ squarefree.
\begin{theorem} \cite{Hooleyarticle}
Let $g \in \mathcal{G}$ and assume GRH$(g)$. Let $h$ be the largest integer $n$ for which $g$ is an $n$-th power and define
$$C(h) \colonequals  \prod \limits_{ \substack{p|h }} \l(1- \frac{1}{p-1} \r)  \prod \limits_{ \substack{p \nmid  h }} \l(1- \frac{1}{p(p-1) } \r).$$
Then, if $g_1 \not \equiv 1 \mod 4$, we have
$$\pi_g(x)= C(h) \frac{x}{\log x}+ O \l( \frac{x \log \log x}{(\log x)^2} \r),$$
and if $g_1 \equiv 1 \mod 4$, we have
$$\pi_g(x)= C(h) \l( 1- \frac{\mu(|g_1|)}{\prod \limits_{\substack{p|h \\ p|g_1}}(p-2) \prod \limits_{\substack{p \nmid h \\ p|g_1}}(p^2-p-1) }    \r)  \frac{x}{\log x}+ O \l( \frac{x \log \log x}{(\log x)^2} \r).$$
\end{theorem}
Lenstra \cite{lenstra77} was the first to show, under assumption of GRH$(g)$, that the primes $p$ with primitive root $g$ and such that $p \equiv a \mod m$, have a natural density in the primes, denoted by $\delta(a,m,g)$. 

\begin{theorem}\cite{lenstra77} \label{thm:lenstra77}
Let $g \in \mathcal{G}$ and assume GRH$(g)$. Let $a$ and $m$ be coprime positive integers, and let $\sigma_a$ be the automorphism on $\mathbb{Q}(\zeta_m)$ determined by $\sigma_a(\zeta_m)= \zeta_m^a$. For a positive integer $n$, let $c_a(n)$ be 1 if $\sigma_a$ restricted to the field $\mathbb{Q}(\zeta_m) \cap \mathbb{Q}(\zeta_n,\sqrt[n]{g} )$ is the identity, and 0 else. Then for
$$\delta(a,m,g) \colonequals \sum \limits_{n=1}^{\infty}  \frac{\mu(n) c_a(n)}{[\mathbb{Q}(\zeta_m,\zeta_n, \sqrt[n]{g}): \mathbb{Q}]},$$
we have
$$\pi_g(x;m,a)= \delta(a,m,g) \frac{x}{\log x} +o \l( \frac{x}{\log x} \r).$$
\end{theorem}
As mentioned earlier, more work was done by Lenstra, Moree and Stevenhagen to study the density $\delta(a,m,g)$. We will use the following result by Moree.
\begin{theorem} \cite{moree2008} \label{thm:moree2008}
Let $g \in \mathcal{G}$ and assume GRH$(g)$. Let $h$ be the largest integer $n$ for which $g$ is an $n$-th power. Let $a$ and $m$ be positive coprime integers. Denote the discriminant of the quadratic number field $\mathbb{Q}(\sqrt{g})$ by $\Delta$, and write $b \colonequals \frac{\Delta}{(m,\Delta)}$, and
$$\gamma \colonequals \begin{cases}
(-1)^{ \frac{b-1}{2} } (m,\Delta), & \text{if $b$ is odd,}\\
1, & \text{otherwise.}
\end{cases}$$
Writing
$$A(a,m,h) \colonequals  \begin{cases} \prod \limits_{ \substack{p|(a-1,m)  }} \l(1- \frac{1}{p} \r) \prod \limits_{ \substack{p \nmid m  \\ p|h }} \l(1- \frac{1}{p-1 } \r)  \prod \limits_{ \substack{p \nmid m \\ p \nmid  h }} \l(1- \frac{1}{p(p-1) } \r), & \text{if } (a-1,m,h)=1,\\
0, & \text{else,} 
\end{cases}$$
we have
\begin{align} \label{formuledelta}
\delta(a,m,g)= \frac{A(a,m,h)}{\phi(m)} \l( 1 + \l( \frac{\gamma}{a} \r) \frac{\mu(|2b|)}{ \prod \limits_{\substack{p | b \\ p|h}} (p-2)\prod \limits_{\substack{p | b \\ p \nmid h}} (p^2-p-1) } \r).
\end{align}
\end{theorem}
If the density $\delta(a,m,g)$ is positive, then it satisfies $\delta(a,m,g) \gg \frac{1}{\phi(m) \log \log m}$, as we now explain. We have
$$\prod \limits_{ \substack{p|(a-1,m)  }} \l(1- \frac{1}{p} \r) \geq \prod \limits_{ \substack{p|m  }} \l(1- \frac{1}{p} \r)= \frac{\phi(m)}{m} \gg \frac{1}{\log \log m},$$
while the other factors from $A(a,m,h)$ are bounded from below by a finite or absolutely convergent, positive product. Furthermore, the factor 
$$1 + \l( \frac{\gamma}{a} \r) \frac{\mu(|2b|)}{ \prod \limits_{\substack{p | b \\ p|h}} (p-2)\prod \limits_{\substack{p | b \\ p \nmid h}} (p^2-p-1) } $$
is bounded from below by $\frac{2}{3}$, since we assumed that $\delta(a,m,g)$ is positive.

\subsection{A Diophantine application}
 A classical example of an application of the Siegel--Walfisz theorem is to count the number of primes $p \leq x$ for which $p-1$ is squarefree, see for example Theorem 11.22 in \cite{MontgomeryVaughan}.  
Inspired by this, we will apply Theorem \ref{thm:a.p.'s in a.p.'ssecondversion} to count primes $p$ with primitive root $g$, for which $ap+b$ is squarefree, where $a >0$ and $b$ are integers with $\gcd(a,b)=1$. We denote by $\pi_{g,a,b}(x)$ the number of such primes bounded by $x$.

\begin{notation} \label{u_0(d)}
For $d$ a non-zero integer with $\gcd(a,d)=1$, denote the solution to the equation
$$ax \equiv -b \mod d^2$$ 
by $u_0(d) \mod d^2$. 
\end{notation}
\begin{definition} \label{deltanatural}
For $\gcd(a,m)=1$ and $g \in \mathcal{G}$, let
$$\delta^{\natural}(a,m,g) \colonequals \frac{A(a,m,h)}{\phi(m)C(h)}.$$
\end{definition}
In the case that $g_1 \equiv 2 \mod 4$, and $8 \nmid m$ or $g_1 \equiv 3 \mod 4$ and $4 \nmid m$, then one can derive from Theorem \ref{thm:moree2008} that $\delta^{\natural}(a,m,g)= \frac{\delta(a,m,g)}{C(h)}$, as is described in Theorem 1.3 from \cite{moree2008}. This means that $\delta^{\natural}(a,m,g)$ is the density of the primes $p$ with primitive root $g$ satisfying \\ $p \equiv a \mod m$, in the set of primes with primitive root $g$. In the other cases, the factor 
$$1 + \l( \frac{\gamma}{a} \r) \frac{\mu(|2b|)}{ \prod \limits_{\substack{p | b \\ p|h}} (p-2)\prod \limits_{\substack{p | b \\ p \nmid h}} (p^2-p-1) }$$
appearing in \eqref{formuledelta} is non-zero, and $\delta^{\natural}(a,m,g)$ does not take into account this correction factor. However, in our application $\delta^{\natural}(a,m,g)$ turns out to be a convenient quantity to consider. See \cite{correction} for an interpretation and the history of the correction factor in Artin's primitive root conjecture.
\begin{theorem} \label{ap+b} 
Let $g \in \mathcal{G}$, and let $a$ and $b$ be non-zero coprime integers, with $a>0$. Let $\Delta$ be the discriminant of $\mathbb{Q}(\sqrt{g})$. Under assumption of GRH$(L,g)$, we then have
\begin{align} \label{ap+bformula}
\begin{split}
\pi_{g,a,b}(x) &= \l( \sum \limits_{\substack{ c \mod |\Delta| \\ \forall p: p| \Delta \implies p^2 \nmid ac+b}} \delta(c, |\Delta|,g) \r)  \l( \prod \limits_{ \substack{ p \nmid \Delta \\ p \nmid ab \\ p \nmid (a+b,h)  }} \l(1- \delta^{\natural}(u_0(p),p^2,g) \r)  \r) \frac{x}{\log x} \\&+ O \l( a (\log |g|) \frac{x \log \log x}{(\log x)^2} \r).
\end{split}
\end{align}
\end{theorem}
\begin{remark}
Only when $\gcd(a,b)$ is not squarefree, there is an obvious reason why $ap+b$ cannot be squarefree. However, we assume that $a$ and $b$ are coprime, in order to avoid some complications in the proof.
\end{remark}
We now provide an informal probabilistic interpretation of Theorem \ref{ap+b}. The factors $1- \delta^{\natural}(u_0(p),p^2,g)$ correspond to the probabilities that for $q$ a prime with $g$ as a primitive root, $aq+b$ is not divisible by $p^2$. However, the product of these factors does not take into account the correction factor in Artin's primitive root conjecture, nor the primes dividing $\Delta$, and we need a factor for the probability that a prime has primitive root $g$. This extra information is contained in the sum of terms $\delta \l(c, |\Delta|,g \r)$. Here $c$ ranges over residue classes modulo $|\Delta|$ such that $ac+b$ is squarefree modulo $|\Delta|$, by which we mean that for no prime factor $p$ of $|\Delta|$, $p^2$ divides $ac+b$.  \\\\
In the first part of the proof of Theorem \ref{ap+bformula}, we imitate Walfisz's method \cite{walfisz} to determine the number of ways to write an integer as a sum of a squarefree number and a prime. However, we use Theorem \ref{thm:a.p.'s in a.p.'ssecondversion} instead of the Siegel--Walfisz theorem. The rest of the proof of Theorem \ref{ap+b} is more specific for this problem, involving calculations related to the density $\delta(a,m,g)$.
\begin{corollary} \label{cor:ap+b}
In the setting of Theorem \ref{ap+b}, assume there is a $c \mod |\Delta|$ such that for every prime $q$ dividing $\Delta$ we have $q^2 \nmid ac+b$, and such that there are infinitely many primes $p$ that satisfy $p \equiv c \mod |\Delta|$ and $\langle g \rangle = \mathbb{F}_p^*$. Then there are infinitely many primes $p$ with $\langle g \rangle = \mathbb{F}_p^*$ such that $ap+b$ is squarefree.
\end{corollary}
Informally speaking this means that if there is an $x \mod |\Delta|$ such that there are infinitely many primes $p$ with primitive root $g$ and $p \equiv x \mod |\Delta|$, and such that $ax+b$ is squarefree modulo $|\Delta|$, then $ap+b$ is squarefree for infinitely many primes $p$ having primitive root $g$.
\begin{proof}
At the end of the proof of Theorem \ref{ap+b} we will see that the product in \eqref{ap+bformula} with factors $1-\delta^{\natural}(u_0(p),p^2,g)$ is positive. Now the corollary follows from the fact that the density $\delta(c,m,g)$ is positive if and only if there are infinitely many primes $p \equiv c \mod m$ with primitive root $g$, which is proved in \cite{lenstra77}
\end{proof}

\section{Proof of Theorem \ref{thm:a.p.'s in a.p.'ssecondversion}}
\subsection{Preparations for the proof of Theorem \ref{thm:a.p.'s in a.p.'ssecondversion}}
The following lemma provides us with a useful tool to detect whether a prime $p$ has $g$ as a primitive root. Its proof is standard.
\begin{lemma} \label{primroot}
Let $p$ be an odd prime. Then the following assertions are equivalent.\\
\textbf{i)} $g$ is a primitive root modulo $p$;\\
\textbf{ii)} for every prime divisor $q$ of $p-1$ we have $g^{\frac{p-1}{q}} \not \equiv 1 \mod p$;\\
\textbf{iii)} for every prime divisor $q$ of $p-1$ there is no $x \in \mathbb{Z}$ such that $x^q \equiv g \mod p$.
\end{lemma}

\begin{definition}
We let $\mathfrak{P}_1$ be the set of all primes, and for any prime $q$ and $g \in \mathcal{G}$ we define
$$\mathfrak{P}_{q,g} \colonequals \l\{ p: q|p-1 \text{ and } g^{\frac{p-1}{q}} \equiv 1 \mod p \r\} . $$
For $k$ a positive squarefree integer, we let
$$\mathfrak{P}_{k,g} \colonequals\bigcap \limits_{q|k}\mathfrak{P}_{q,g}. $$
\end{definition}
We see that $\langle g \rangle = \mathbb{F}_p^*$ if and only if $p \notin \mathfrak{P}_{q,g}$ for every prime $q$.
If $p \in \mathfrak{P}_{q,g}$ and $p \leq x$, then $q|p-1$ and thus $q \leq x-1$, therefore
$$\pi_g\l(x;m,a\r) = \# \{ p \leq x: p \equiv a \mod m, \forall q \leq x-1: p \notin \mathfrak{P}_{q,g} \}.$$
The proof of the next lemma is also standard.
\begin{lemma} \label{intersection}
For $k \in \mathbb{N}$ squarefree, we have 
$\mathfrak{P}_{k,g} = \l\{ p: k|p-1 \text{ and } g^{\frac{p-1}{k}} \equiv 1 \mod p \r\} . $
\end{lemma}
We further introduce some quantities counting primes in arithmetic progressions, having certain properties related to $g$ being a primitive root.
\begin{definition}
Let $g \in \mathcal{G}$, $m \in \mathbb{N}$ and let $1 \leq a \leq m$ be coprime to $m$.\\
Let, for $x, \eta \in \mathbb{R}_{>0}$,
$$N_{a,m,g}(x, \eta) \colonequals \# \l\{p \leq x: p \equiv a \mod m \text{ and } \forall ~ q \leq \eta: p \notin \mathfrak{P}_{q,g} \r\}.$$
For $x, \eta_1, \eta_2 \in \mathbb{R}_{>0}$ with $\eta_2 > \eta_1$, let
$$M_{a,m,g}(x, \eta_1,\eta_2) \colonequals  \# \{p \leq x: p \equiv a \mod m \text{ and } \exists ~ \eta_1 <q \leq \eta_2: p \in \mathfrak{P}_{q,g} \}.$$
For $k \in \mathbb{N}$ squarefree and $x>0$, let
$$P_{k,g}(x;m,a) \colonequals \# \{p \leq x: p \equiv a \mod m \text{ and } p \in \mathfrak{P}_{k,g} \}.$$
\end{definition}
Throughout the proof, we assume that $m \ll x^{\epsilon}$ for every $\epsilon>0$. We can assume this without loss of generality, since for $m > x^{\epsilon}$, the main term in Theorem \ref{thm:a.p.'s in a.p.'ssecondversion} is smaller than the error term, and the theorem is trivially true.
\begin{notation}
Throughout the proof, the notations $\ll$ and $\gg$ indicate an estimate with absolute implied constant, unless mentioned otherwise.
\end{notation}

\subsection*{Proof of Theorem \ref{thm:a.p.'s in a.p.'ssecondversion}}

Let $0 < \xi_1<\xi_2<\xi_3<x-1$ be functions that satisfy \\ $\lim \limits_{x \to \infty} \xi_i(x)= \infty ~ (i=1,2,3)$. We will choose the functions $\xi_i$ at the end of the proof. By an elementary counting argument we see that
\begin{align}
N_{a,m,g}(x,\xi_1) -\pi_g(x;m,a)\ll \sum \limits_{i=1}^2 M_{a,m,g}(x, \xi_i,\xi_{i+1}) + M_{a,m,g}\l(x, \xi_3,x-1\r).  \label{identitystart}
\end{align} 
The proof consists of finding upper bounds for the terms on the right side, and finding an asymptotic with explicit error term for $N_{a,m,g}(x,\xi_1)$.
\subsection{Bounding $M_{a,m,g}(x, \xi_3,x-1)$}

If $p$ is counted by $M_{a,m,g}(x, \xi_3,x-1)$, then there is a $\xi_3 < q \leq x-1$ such that $q|(p-1)$ and $p|\l(g^{2\frac{p-1}{q}}-1\r)$. Because 
$\frac{p-1}{q} \leq \frac{x}{\xi_3}$, we see that $p$ divides the product $\prod \limits_{1 \leq t \leq \frac{x}{\xi_3}} \l( g^{2t}-1 \r)$. Therefore,
\begin{align*}
 M_{a,m,g}(x, \xi_3,x-1) \leq  \# \Bigg\{ p : p \equiv a \mod m \text{ and } p| \prod \limits_{1 \leq t \leq \frac{x}{\xi_3}} \l( g^{2t}-1 \r) \Bigg\}.
\end{align*}
Letting $p_1<p_2<...<p_l$ be counted on the right side, we get

\begin{align*}
\prod \limits_{1 \leq t \leq \frac{x}{\xi_3}} \l( g^{2t}-1 \r) \geq p_1p_2 \cdots p_l
\geq a(a+m)(a+2m) \cdots (a+(l-1)m)
\geq m^{l-1}.
\end{align*}
It follows that $(l-1) \log m$ is at most
\begin{align}
\begin{split} \label{log}
 \sum \limits_{1 \leq t \leq \frac{x}{\xi_3}} \log \l( g^{2t}-1  \r)
\ll\log (|g|) \frac{x^2}{{\xi_3}^2}.
\end{split}
\end{align}
Combining the estimates so far yields the next result.
\begin{lemma} \label{lem:sum1}
The following holds with an absolute implied constant,
\begin{align*}
M_{a,m,g}(x, \xi_3,x-1) \ll \log (|g|) \frac{x^2}{{\xi_3}^2 \log (2m)} +1.
\end{align*}
\end{lemma}
\subsection{Bounding $M_{a,m,g}(x, \xi_2,\xi_3)$} Throughout this subsection we assume that $m \xi_3 \leq x$ and $\xi_3 \geq \e^{\e} \xi_2$.
It is obvious that
\begin{align}
\begin{split} \label{S_3}
M_{a,m,g}(x, \xi_2,\xi_3) \leq \sum \limits_{\xi_2 < q \leq \xi_3}  \# \{ p \leq x: p \equiv a \mod m \text{ and } p \equiv 1 \mod q \}.
\end{split}
\end{align}
The two progressions can be combined into a single progression modulo $[m,q]$. Using the Brun--Titchmarsh theorem, one gets a good upper bound if $[m,q]$ is large, which happens when $q \nmid m$. The cases $q|m$ give a bad upper bound but we will see that they are rare.
\\\\
\textbf{Case i)} Assume $q|m$. Then the summand in \eqref{S_3} equals $\pi(x;m,a)$ if $a \equiv 1 \mod q$, and $0$ otherwise. Hence, by the Brun--Titchmarsh theorem, the number of primes $p \leq x$ congruent to $a \mod m$ and $1 \mod q$ is
$$
\ll \min \Bigg\{ \frac{x}{\phi(m) \log \l( \frac{2x}{m} \r)}, \frac{x}{m} \Bigg\} \leq \frac{x}{\phi(m) \log \l( \frac{2x}{m} \r)}.$$
\textbf{Case ii)} Assume $q \nmid m$. Then by the Chinese remainder theorem the progressions \\ $p \equiv a \mod m$ and $p \equiv 1 \mod q$ can be combined into a single progression modulo $mq$.
Again applying the Brun--Titchmarsh theorem shows that in this case the number of primes $p \leq x$ congruent to $a \mod m$ and $1 \mod q$ is
\begin{align*}
\ll \min \Bigg\{ \frac{x}{\phi(mq) \log \l( \frac{2x}{mq} \r)}, \frac{x}{mq}+1 \Bigg\} \ll \frac{x}{q \phi(m)\log \l( \frac{2x}{mq} \r)}.
\end{align*} 

Applying the above cases to \eqref{S_3} we can bound $M_{a,m,g}(x, \xi_2,\xi_3)$ by
\begin{align} \label{S_3estimate2}
\begin{split}
\ll \sum \limits_{\substack{\xi_2 < q \leq \xi_3 \\ q | m}} \frac{x}{\phi(m) \log \l( \frac{2x}{m} \r)} + \sum \limits_{\substack{\xi_2 < q \leq \xi_3 \\ q \nmid m}}  \frac{x}{q \phi(m)  \log \l( \frac{2x}{mq} \r)}.
\end{split}
\end{align} 
The number of prime divisors of $m$ which are greater than $\xi_2$, is at most $\frac{\log m}{\log \xi_2}$, hence
\begin{align} \label{part1}
\sum \limits_{\substack{\xi_2 < q \leq \xi_3 \\ q | m}} \frac{x}{\phi(m) \log \l( \frac{2x}{m} \r)} \ll \frac{\log m}{\log \xi_2} 
\frac{x}{\phi(m) \log \l( \frac{2x}{m} \r)}.
\end{align} 
By Mertens' estimate, the second sum in \eqref{S_3estimate2} is bounded by
\begin{align} \label{part3}
\begin{split}
\ll  \frac{x }{ \phi(m)\log \l( \frac{2x}{m\xi_3}  \r)} \sum \limits_{\substack{\xi_2 < q \leq \xi_3 }}  \frac{1}{q}
\ll\frac{x}{\phi(m) \log \l( \frac{2x}{m\xi_3}  \r)} \log \l( \frac{\log \xi_3}{\log \xi_2} \r) .
\end{split}
\end{align}
Combining the estimates \eqref{S_3estimate2}, \eqref{part1} and \eqref{part3}, we arrive at
\begin{align} \label{waseerstlemma}
M_{a,m,g}(x, \xi_2,\xi_3) \ll  \frac{\log m}{\log \xi_2} 
\frac{x}{\phi(m) \log \l( \frac{2x}{m} \r)}
+ \frac{x \log \l( \frac{\log \xi_3}{\log \xi_2} \r)}{  \phi(m) \log \l( \frac{2x}{m \xi_3} \r)  }.
\end{align}
\begin{lemma}  \label{sum 3 small m}
Assume that $\xi_3 \geq \e^{\e} \xi_2$, and that there is an $0<A<1$ such that \\ $m \xi_3 \leq x^{1-A}$. Then
$$M_{a,m,g}(x, \xi_2,\xi_3) \ll \frac{x}{\phi(m) \log x} \l( \frac{\log  m }{\log \xi_2} + \log \l( \frac{\log \xi_3}{\log \xi_2} \r)  \r),$$
where the implied constant depends at most on $A$.
\end{lemma}
\begin{proof}
Due to the assumptions of the lemma, we have $\log \l( \frac{2x}{m} \r) \gg \log (\frac{2x}{m \xi_3} ) \gg \log x$. Combining this with \eqref{waseerstlemma} yields the required estimate.
\end{proof}
\subsection{An asymptotic formula for $P_{k,g}(x;m,a)$}
For the investigation of the terms \\ $M_{a,m,g}(x, \xi_1,\xi_2)$ and $N_{a,m,g}(x, \xi_1)$, we need to find an asymptotic formula for $P_{k,g}(x;m,a)$. Recall the definitions for $\Delta$, $h$ and $b$ from Theorem \ref{thm:moree2008}, and for $\zeta_m$ and $\sigma_a$ from Theorem \ref{thm:lenstra77}. 
\begin{notation}
For $k,m \in \mathbb{N}$ with $k$ squarefree, we write $F_{m,k}:= \mathbb{Q}(\zeta_m, \zeta_k, \sqrt[k]{g})$. For a prime number $p$, let $(p, F_{m,k}/ \mathbb{Q})$ be the Artin symbol in the Galois group of $F_{m,k}$ over $\mathbb{Q}$.
\end{notation}

\begin{lemma} \label{pinpkg}
Let $p$ be a prime, $a,m, k \in \mathbb{N}$ with $k$ squarefree and $(a,m)=1$. Then we have $p \equiv a \mod m$ and $p \in \mathfrak{P}_{k,g}$ if and only if
$$(p, F_{m,k} ) |_{ \mathbb{Q}(\zeta_m)} =\sigma_a \text{ and } (p, F_{m,k} ) |_{F_{k,k} } = \text{id}.$$ 
\end{lemma}
\begin{proof}
The Artin symbol $(p, F_{m,k} ) $ restricted to $ \mathbb{Q}(\zeta_m)$ equals $\sigma_a$ if and only if \\ $p \equiv a \mod m$. Further, $(p, F_{m,k} ) $ is the identity on $\mathbb{Q}(\zeta_k)$ if and only if\\  $p \equiv 1 \mod k$, and it is the identity on $\mathbb{Q}(\sqrt[k]{g})$ if and only if $g^{\frac{p-1}{k}} \equiv 1 \mod p$. Combining these observations with Lemma \ref{intersection} completes the proof.
\end{proof}
We use the following three lemma's from Frei, Koymans and Sofos \cite{sofosfreikoymans}.
\begin{lemma} \label{degree} (Lemma 2.2 from \cite{sofosfreikoymans})
For $k,m \in \mathbb{N}$ with $k$ squarefree we have
$$[F_{m,k}: \mathbb{Q}]= \frac{k_1 \phi([m,k])}{\epsilon(m,k)},$$
where $k_1 \colonequals \frac{k}{(k,h)}$, and
$$\epsilon(m,k) \colonequals \begin{cases}
2 & \text{if } 2|k \text{ and } \Delta | [m,k],\\
1, & \text{else.}
\end{cases}$$
\end{lemma}

\begin{lemma} \label{automorphism} (Lemma 2.5 from \cite{sofosfreikoymans})
For $m \in \mathbb{N}$ and $k \in \mathbb{N}$ squarefree, there is an automorphism $\sigma \in \mathrm{Aut} \l( F_{m,k}\r)$ with 
$$ \sigma|_{  \mathbb{Q}(\zeta_m) } = \sigma_a \text{ and } \sigma|_{F_{k,k}}= \mathrm{id}$$
if and only if
\begin{align} \label{conditions}
\begin{split}
& a \equiv 1 \mod (m,k), \text{ and} \\
& 2|k, \Delta \nmid k, \Delta | [m,k]  \text{ implies that } \l( \frac{\gamma}{a} \r) =1.
\end{split}
\end{align} 
If $\sigma$ exists, then it is unique and in the centre of $\mathrm{Aut} \l( F_{m,k} \r)$.
\end{lemma}
The following lemma is a slightly weaker formulation of Lemma 2.3 from \cite{sofosfreikoymans}.
\begin{lemma} \label{discbound}
We have
$$\frac{\log | \Delta \l( F_{m,k} \r) |}{[F_{m,k}: \mathbb{Q}] } \leq \log k + \log([m,k]) + 2\log |g|,$$
where $\Delta(F_{m,k})$ denotes the discriminant of $F_{m,k}$.
\end{lemma}

\begin{lemma} \label{algebra} 
Let $g \in \mathcal{G}$, $a,m \in \mathbb{N}$ coprime and $k \in \mathbb{N}$ squarefree. Then we have
$$\{p: p \equiv a \mod m, p \in \mathfrak{P}_{k,g} \} \neq \emptyset$$
if and only if the conditions in \ref{conditions} are satisfied.
Assuming GRH$(L,g)$, we have in this case
$$P_{k,g}(x;m,a)= \frac{\epsilon(m,k)}{k_1 \phi([m,k])} \mathrm{li}(x) + O \l( \sqrt{x} \l( \log [m,k] + \log |g| + \log x \r) \r).$$
\end{lemma}
\begin{proof}
The first claim follows from Lemma \ref{pinpkg} and Lemma \ref{automorphism}. Now assume the conditions in \ref{conditions} are satisfied. Then the asymptotic formula for $P_{k,g}(x;m,a)$ follows from Serre's version of the Chebotarev density theorem \cite{Serre}, under assumption of GRH$(L,g)$, in combination with Lemma \ref{discbound}.
\end{proof}

\subsection{Bounding $M_{a,m,g}(x, \xi_1,\xi_2)$}
Using the notation from Lemma \ref{degree}, we have $\frac{1}{q_1} \ll \frac{\log |g|}{q}$ for each prime $q$. Thus,
by Lemma \ref{algebra} we can bound $M_{a,m,g}(x, \xi_1,\xi_2)$ by
\begin{align*}
\ll \sum  \limits_{\xi_1 < q \leq \xi_2} P_{q,g}(x;m,a)
\ll  \sum  \limits_{\xi_1 < q \leq \xi_2} \l( \log |g| \frac{\text{li}(x)}{q \phi([m,q])} + \sqrt{x} \l( \log x + \log |g| \r) \r)  .
\end{align*}
Using $\pi(\xi_2) \ll \frac{\xi_2}{\log \xi_2}$, the second term makes the following contribution,
\begin{align*}
\ll \frac{\xi_2}{\log \xi_2} \sqrt{x} \log x \log |g|.
\end{align*} 
In the first term, the many terms with $q \nmid m$ have a good upper bound, and the fewer terms with $q|m$ have a bad upper bound. Namely, 
\begin{align*}
\sum  \limits_{\substack{\xi_1 < q \leq \xi_2\\ q \nmid m}} \l(  \frac{1}{q \phi([m,q])} \r)  \ll  \frac{1}{\phi(m)} \sum  \limits_{\substack{\xi_1 < q \leq \xi_2}}   \frac{1}{q^2} \ll \frac{1}{\phi(m) } \frac{1}{\xi_1 \log \xi_1},
\end{align*}
where the estimate $\sum \limits_{q > \alpha} \frac{1}{q^2} \ll \frac{1}{\alpha \log \alpha}$ can be easily proved via partial summation.
Because $m$ has at most $\frac{\log m}{\log \xi_1}$ prime divisors $q>\xi_1$, we see that
\begin{align*}
\sum  \limits_{\substack{\xi_1 < q \leq \xi_2\\ q | m}} \frac{1}{q \phi([m,q])}  &= \frac{1}{\phi(m)} \sum  \limits_{\substack{\xi_1 < q \leq \xi_2\\ q | m}} \frac{1}{q } \leq \frac{1}{\phi(m)}\frac{\log m}{\log \xi_1} \frac{1}{\xi_1}.
\end{align*}

\begin{lemma} \label{lem:sum2}
We have
$$M_{a,m,g}(x, \xi_1,\xi_2) \ll \log |g| \frac{\log m}{\phi(m) \xi_1 \log \xi_1} \frac{x}{\log x} +  \log |g| \sqrt{x} \log x \frac{\xi_2}{\log \xi_2}.$$
\end{lemma}

\subsection{The main term $N_{a,m,g}(x, \xi_1)$}
\begin{notation}
For $k \in \mathbb{N}$, let $P^+(k)$ be the largest prime divisor of $k$.
\end{notation}
In this section we assume that $4^{\xi_1} \leq x$.
By the inclusion-exclusion principle and by Lemma \ref{intersection} we have
$$N_{a,m,g}(x, \xi_1)= \sum \limits_{\substack{k \in \mathbb{N} \\ P^+(k) \leq \xi_1}} \mu(k)P_{k,g}(x;m,a). $$
Every $k$ in the sum can be bounded as
$$k \leq \prod \limits_{p \leq \xi_1} p \leq 4^{\xi_1} \leq x.$$
Applying Lemma \ref{algebra} and using the bound for $k$ in the error term we get
\begin{align} \label{maintermstart}
N_{a,m,g}(x, \xi_1)=  \frac{\li (x)}{\phi(m)} \sum \limits_{\substack{k \in \mathbb{N} \\ P^+(k) \leq \xi_1}} f(k)  + O \l(4^{\xi_1} \sqrt{x} \log x \log |g| \r),
\end{align}
where $f$ is defined as
$$f(k) \colonequals
\begin{cases}
 \frac{\mu(k) \epsilon(m,k) \phi((m,k))}{k_1 \phi(k)},& \text{if \eqref{conditions} holds,}\\
 0, & \text{else.}
 \end{cases}$$
 Because of the factor $\epsilon(m,k)$, the function $f$ is not multiplicative, and the sum in \eqref{maintermstart} cannot be written as an Euler product. Instead, we write
\begin{align} \label{new1}
\sum \limits_{\substack{k \in \mathbb{N} \\ P^+(k) \leq \xi_1}} f(k) = \sum \limits_{k=1}^{\infty} f(k) - \sum \limits_{\substack{k \in \mathbb{N} \\ P^+(k) > \xi_1}} f(k). 
\end{align}
We will bound the second sum, and for the first sum we have the following lemma.
\begin{lemma} \label{delta}
We have
$$ \frac{1}{\phi(m)} \sum \limits_{k=1}^{\infty} f(k)= \delta(a,m,g).$$
\end{lemma}
\begin{proof}
Recall from Theorem \ref{thm:lenstra77} the definition for $c_a(k)$, and the expression for $\delta(a,m,g)$.
As stated in the proof of Theorem 1.2 from \cite{moree2008} we have $c_a(k)=1$ if and only if \eqref{conditions} holds, hence
\begin{align*} 
\frac{1}{\phi(m)}f(k)=
 \frac{\mu(k) \epsilon(m,k) c_a(k) }{k_1 \phi([m,k])}.
\end{align*} 
Using Lemma \ref{degree} we get
\begin{align*}
\frac{1}{\phi(m)} f(k)= \frac{ \mu(k) c_a(k)}{[\mathbb{Q}(\zeta_m,\zeta_k, \sqrt[k]{g}): \mathbb{Q}]}.
\end{align*}
\end{proof}
In order to bound the second sum in \eqref{new1}, we note that
\begin{align} \label{new2}
\sum \limits_{\substack{k \in \mathbb{N} \\ P^+(k) > \xi_1}} f(k) \ll \log |g| \sum \limits_{\substack{k \in \mathbb{N} \\ P^+(k) > \xi_1}} \frac{\mu(k)^2 \phi((m,k))}{k \phi(k)} .
\end{align}
Using Euler products we get
\begin{align} \label{new3}
\begin{split}
\sum \limits_{\substack{k \in \mathbb{N} \\ P^+(k) > \xi_1}} \frac{\mu(k)^2 \phi((m,k))}{k \phi(k)}  &= \sum \limits_{k =1}^{\infty} \frac{\mu(k)^2 \phi((m,k))}{k \phi(k)}  - \sum \limits_{\substack{k \in \mathbb{N} \\ P^+(k) \leq  \xi_1}} \frac{\mu(k)^2 \phi((m,k))}{k \phi(k)} 
\\&=\prod \limits_{p} \l( 1 + \frac{\phi((m,p))}{p(p-1)} \r) - \prod \limits_{p \leq \xi_1} \l( 1 + \frac{\phi((m,p))}{p(p-1)} \r)
\\&= \prod \limits_{p} \l( 1 + \frac{\phi((m,p))}{p(p-1)} \r) \l( 1 -  \prod \limits_{p > \xi_1} \frac{1}{1 + \frac{\phi((m,p))}{p(p-1)} } \r).
\end{split}
\end{align}
Now using $\log \l( \frac{1}{1+x} \r) = O(x)$ when $|x| \leq \frac{1}{2}$, we get
\begin{align*}
\log \l( \prod \limits_{p> \xi_1} \frac{1}{1+ \frac{\phi((m,p))}{p(p-1)} } \r) & \ll \sum \limits_{p > \xi_1} \frac{\phi((m,p))}{p(p-1)} 
 \\&=  \sum \limits_{\substack{p > \xi_1 \\ p|m}} \frac{1}{p} +  \sum \limits_{\substack{p > \xi_1 \\ p \nmid m}} \frac{1}{p(p-1)}  
\ll  \frac{\log(2m)}{\xi_1 \log \xi_1}.
\end{align*}
Because of the estimate $e^x=1+O(x)$, when $|x| \leq \frac{1}{2}$, we see that
\begin{align}\label{new4}
\prod \limits_{p> \xi_1} \frac{1}{1+ \frac{\phi((m,p))}{p(p-1)} } =  1+ O \l( \frac{\log(2m)}{\xi_1 \log \xi_1} \r).
\end{align}
The other product is estimated in the following way,
\begin{align} \label{new5}
\begin{split}
 \prod \limits_{p} \l( 1 + \frac{\phi((m,p))}{p(p-1)} \r)  &\ll \prod \limits_{p|m} \l( 1 + \frac{1}{p} \r)  
 \\&\ll \frac{m}{\phi(m)} \ll \log \log (10m).
 \end{split}
\end{align}
Next one can combine \eqref{new2} up to \eqref{new5} to bound $\sum \limits_{\substack{k \in \mathbb{N} \\ P^+(k) > \xi_1}} f(k)$. One combines the acquired upper bound with \eqref{maintermstart}, \eqref{new1} and Lemma \ref{delta} to obtain the following result.

%
%
%
\begin{lemma} \label{finalmainterm}
Assume that $4^{\xi_1} \leq x$. Then
\begin{align*}
N_{a,m,g}(x, \xi_1)
&= \delta(a,m,g) \li(x) + O\l( \frac{x \log |g|}{\phi(m) \log x}  \frac{\log (2m) \log \log (10m)}{\xi_1 \log \xi_1} \r)  \\&+ O( \log |g| 4^{\xi_1} \sqrt{x} \log x).
\end{align*}
\end{lemma}

\subsection{Conclusion of the proof of Theorem \ref{thm:a.p.'s in a.p.'ssecondversion}}
Combining \eqref{identitystart} with the Lemma's \ref{lem:sum1}, \ref{sum 3 small m}, \ref{lem:sum2}, and \ref{finalmainterm} yields the next result.
\begin{proposition} \label{firstversion}
If we have $m \xi_3 \leq x$, $\xi_3 \geq \e^{\e} \xi_2$ and $4^{\xi_1} \leq x$, and there is an $0<A<1$ such that $m \leq \frac{x^{1-A}}{\xi_3}$, then $\pi_g(x;m,a)- \delta(a,m,g) \li(x)$ is
\begin{align}
 \ll &4^{\xi_1} \sqrt{x} \log x \log |g|+  \frac{x \log |g|}{\phi(m) \log x}  \frac{\log (2m) \log \log (10m)}{\xi_1 \log \xi_1} \label{regel1}
\\&  +\frac{x^2 \log |g|}{{\xi_3}^2 \log (2m)}  +   \sqrt{x} \log x \frac{\xi_2 \log |g|}{\log \xi_2}\label{regel2}
\\&+ \frac{x}{\phi(m) \log x} \l( \log \l( \frac{\log \xi_3}{\log \xi_2} \r) + \frac{\log  m }{\log \xi_2} \r)\label{regel3}
\end{align} 
\end{proposition}
Now we choose our functions $\xi_i$ as follows,
$$\xi_1= \frac{1}{6} \log x, ~~ \xi_2= \sqrt{x}(\log x)^{-2} m^{-1} ,~~\xi_3= \sqrt{x}(\log x)^{2}m.$$ 
Note that the conditions from Proposition \ref{firstversion} are satisfied, since we assumed that $m \ll x^{\epsilon}$ for every $\epsilon>0$. We conclude the proof of Theorem \ref{thm:a.p.'s in a.p.'ssecondversion} by bounding all the error terms in Proposition \ref{firstversion} from above by $\log |g| \frac{x }{\phi(m) \log x} \frac{\max\l\{ \log \log x, \log (2m) \r\}}{\log x }$, the error term in Theorem \ref{thm:a.p.'s in a.p.'ssecondversion}.
\subsubsection{Error terms in \eqref{regel1}}
Note that
$4^{\xi_1}  \sqrt{x} \log x =  x^{\frac{1}{6} \log 4 + \frac{1}{2}} \log x \ll  x^{0.9}.$
Using the estimates $\log \log (10m) \ll \log \log x$ and $\xi_1 \log \xi_1 \gg \log x \log \log x$, we see that the second term in \eqref{regel1} is 
$$\ll \log |g| \frac{x }{\phi(m) \log x} \frac{\log(2m)}{\log x}.
$$

\subsubsection{Error terms in \eqref{regel2}}
The first term in \eqref{regel2} can easily be estimated as follows,
\begin{align*}
\frac{x^2 \log |g|}{{\xi_3}^2 \log (2m)} 
= \frac{x\log |g|}{(\log x)^4 m^2 \log (2m)} \ll \frac{x\log |g|}{ \phi(m)(\log x)^2} .
\end{align*} 
Using the bounds
 $\xi_2\ll \sqrt{x}(\log x)^{-2} \phi(m)^{-1}$
and $\log \xi_2 \gg \log x$, we obtain
  $$\sqrt{x} \log x \frac{\xi_2 \log |g|}{\log \xi_2} \ll \frac{x \log |g|}{\phi(m) (\log x)^2}.
$$

\subsubsection{Error terms in \eqref{regel3}}
In order to prove that the error term in \eqref{regel3} is bounded by the error term in Theorem \ref{thm:a.p.'s in a.p.'ssecondversion}, it suffices to show that
$$\log \l( \frac{\log \xi_3}{\log \xi_2} \r)  + \frac{\log  m }{\log \xi_2}  \ll \log |g| \frac{\max \l\{ \log \log x, \log m \r\}}{\log x}.$$
Since $\log \xi_2 \gg \log x$, we have
$$\frac{\log  m }{\log \xi_2} \ll \frac{\log m}{\log x}.$$
Furthermore, due to our choice of $\xi_2$ and $\xi_3$ we have
\begin{align*}
\log \l( \frac{\log \xi_3}{\log \xi_2} \r)
= \log \l( \frac{1 + \frac{4\log \log x}{\log x} + \frac{2 \log m}{\log x}}{1 - \l( \frac{4\log \log x}{\log x} + \frac{2 \log m}{\log x} \r)} \r).
\end{align*} 
For $|\epsilon|< \frac{1}{2}$, we have
$\log \l( \frac{1+ \epsilon}{1- \epsilon} \r)\ll \epsilon.$
Choosing $\epsilon= \frac{4\log \log x}{\log x} + \frac{2 \log m}{\log x}$, we see that
$$\log \l( \frac{\log \xi_3}{\log \xi_2} \r) \ll \frac{\log \log x}{\log x} + \frac{\log m}{\log x} \ll \frac{\max \l\{ \log \log x, \log m \r\} }{\log x}.$$
We note that the error term from \eqref{regel3} comes from the hardest case in Hooley's argument, and also gives us the largest error term. Because all the error terms in Proposition \ref{firstversion} are of order $\log |g| \frac{x }{\phi(m) \log x} \frac{\max\l\{ \log \log x, \log (2m) \r\}}{\log x \}}$, we have proved
Theorem \ref{thm:a.p.'s in a.p.'ssecondversion}.

\section{Proof of Theorem \ref{ap+b}} \label{proofapplication}

To detect if a number $n \in \mathbb{N}$ is squarefree, we use the identity
\begin{align*} 
\sum \limits_{\substack{ d\in \mathbb{N} \\ d^2 | n }} \mu(d) = 
\begin{cases}
1, & \text{if $n$ is squarefree,} \\
0, & \text{else.}
\end{cases}
\end{align*}

Recall Notation \ref{u_0(d)}. Applying the above to $n=ap+b$ and interchanging the order of summation, we get
\begin{align} \label{pistart}
\pi_{g,a,b}\l(x\r) = \sum \limits_{\substack{p \leq x \\ \langle g \rangle = \mathbb{F}_p^*}}  \sum \limits_{\substack{d \in \mathbb{N} \\ d^2|ap+b}} \mu\l(d\r)
= \sum \limits_{\substack{d \leq \sqrt{ax+b}\\ (a,d)=1}} \mu\l(d\r)  \pi_g\l(x;d^2,u_0(d)\r).
\end{align}
We split the interval $[1, \sqrt{ax+b}]$ in two parts $[1,y]$ and $(y,\sqrt{ax+b}]$, where $y \colonequals \l(\log x\r)^2$. In \eqref{pistart}, the terms with $d>y$ contribute at most $ O \l( \frac{x}{(\log x)^2}+\sqrt{ax+b} \r)$, because we can bound $\pi_g\l(x;d^2,u_0(d)\r)$ by the number of integers $n \leq x$ with $n \equiv u_0(d) \mod d^2$, which is at most $\frac{x}{d^2}+1$. The terms with $(u_0(d),d)>1$ contribute $O \l( \sqrt{ax+b}  \r)$, because for such $d$ we have $ \pi_g\l(x;d^2,u_0(d)\r) \leq 1$. Hence,
\begin{align} \label{usebounds}
 \pi_{g,a,b}(x)= \sum \limits_{\substack{d \leq y\\ (a,d)=1 \\ (u_0(d),d)=1 }} \mu\l(d\r)   \pi_g\l(x;d^2,u_0(d)\r) + O \l( \frac{x}{\l(\log x\r)^2} + \sqrt{ax+b}  \r).
\end{align}
For $d \leq y$ we have $\log(2d^2) \ll \log \log x$, so by Theorem \ref{thm:a.p.'s in a.p.'ssecondversion} we have 
\begin{align} \label{applysw}
\begin{split}
\sum \limits_{\substack{d \leq y\\ (a,d)=1 \\ (u_0(d),d)=1 }} \mu\l(d\r)   \pi_g\l(x;d^2,u_0(d)\r) &= \frac{x}{\log x} \sum \limits_{\substack{d \leq y\\ (a,d)=1 \\ (u_0(d),d)=1 }} \l( \mu\l(d\r)  \delta(u_0(d),d^2,g) \r) 
\\&+O \l(\log |g|   \frac{ x \log \log x}{(\log x)^2}   \r),
\end{split}
\end{align} 
where we use that
$\sum \limits_{d=1}^{\infty} \frac{1}{\phi(d^2)} =O(1).$ This illustrates the usefulness of Theorem \ref{thm:a.p.'s in a.p.'ssecondversion}; the uniformity of the error term on $d$ ensures that the sum of error terms in \eqref{applysw} can be bounded conveniently.\\\\
Since $ \delta\l(u_0(d),d^2,g\r) \ll \frac{1}{\phi\l(d^2\r)} $, the series $\sum \limits_{\substack{d \geq 1}} \mu\l(d\r)     \delta \l(u_0(d),d^2,g\r)$ is absolutely convergent, and
\begin{align*}
\begin{split}
\sum \limits_{\substack{d > y\\ (a,d)=1  \\ (u_0(d),d)=1}}\mu\l(d\r)    \delta \l(u_0(d),d^2,g\r) &\ll \sum \limits_{d > y} \frac{1}{d\phi(d)} \ll \frac{1}{(\log x)^2}.
\end{split}
\end{align*}
Extending the sum in \eqref{applysw} and substituting the result in \eqref{usebounds} we arrive at
\begin{align} \label{endoffirstpart}
\begin{split}
 \pi_{g,a,b}(x)&= \frac{x}{\log x}   \sum \limits_{\substack{d = 1\\(a,d)=1  \\ (u_0(d),d)=1}}^{\infty} \mu\l(d\r)  \delta \l(u_0(d),d^2,g\r) 
+O \l( a (\log |g|)  \frac{ x \log \log x}{ (\log x)^2 } \r)  .
  \end{split}
\end{align}
So far we have followed Walfisz's method \cite{walfisz} to determine the number of ways to write an integer as a sum of a squarefree number and a prime, except for the fact that we applied Theorem \ref{thm:a.p.'s in a.p.'ssecondversion} instead of the Siegel--Walfisz theorem. The remaining part of the proof consists of calculating the sum appearing in \eqref{endoffirstpart}, and is specific for this problem.\\\\ We start by rewriting the conditions in the sum to see that they behave multiplicavely in $d$, which enables us to use Euler products. Recall from Theorem \ref{thm:moree2008} that if $\delta \l(u_0(d),d^2,g\r)>0$, then $(u_0(d)-1,d,h)=1$.
\begin{lemma} \label{technisch} Let $d \in \mathbb{N}$ be squarefree and coprime to $a$, and let $p$ be a prime divisor of $d$.\\
\textbf{i)} We have $(u_0(d)-1,d,h)=1$ if and only if $(a+b,d,h)=1.$\\
\textbf{ii)} We have $p|u_0(d)$ if and only if $p|b$, and we have $p|(u_0(d)-1)$ if and only if $p|(a+b)$.\\
\textbf{iii)} We have $(u_0(d),d)=(b,d)$.
\end{lemma}
\begin{proof}
\textbf{i and ii)} By definition of $u_0(d)$ we have $a(u_0(d)-1)+a+b \equiv 0 \mod d^2$, hence $p$ divides $ a (u_0(d)-1) +a+b$. Therefore, $p$ divides $u_0(d)-1$ if and only if $p$ divides $a+b$, from which (i) follows. Since $p$ divides $au_0(d)+b$, we also see that $p|u_0(d)$ if and only if $p|b$.\\
\textbf{iii)} Because $d$ is squarefree, it follows from (ii) that
$(u_0(d),d)=(b,d)$.
\end{proof}
With Theorem \ref{thm:moree2008} in mind, we write $b_M(d) \colonequals \frac{\Delta}{(\Delta,d^2)}$ for $\Delta= \Delta \l(\mathbb{Q} \l(\sqrt{g} \r) \r)$, and
$$\gamma(d) \colonequals \begin{cases}
(-1)^{ \frac{b_M(d)-1}{2} } (d^2,\Delta), & \text{if $b_M(d)$ is odd,}\\
1, & \text{otherwise.}
\end{cases}$$
Combining the observations from Lemma \ref{technisch} with Theorem \ref{thm:moree2008} we find
\begin{align} 
&\sum \limits_{\substack{d =1\\(a,d)=1  \\ (u_0(d),d)=1}}^{\infty} \mu\l(d\r)  \delta \l(u_0(d),d^2,g\r) = \sum \limits_{\substack{d =1\\(ab,d)=1  \\ (a+b,d,h)=1}}^{\infty} \mu\l(d\r) \frac{A(u_0(d),d^2,h)}{\phi(d^2)} \label{somdeel1}
\\&+ \sum \limits_{\substack{d =1\\(ab,d)=1  \\ (a+b,d,h)=1}}^{\infty} \mu\l(d\r) \frac{A(u_0(d),d^2,h)}{\phi(d^2)} \l( \frac{\gamma(d) }{u_0(d)} \r) \frac{\mu(2|b_M(d) |)}{ \prod \limits_{\substack{p|b_M(d)  \\ p|h}} (p-2) \prod \limits_{\substack{p|b_M(d)  \\ p \nmid h}} (p^2-p-1)},\label{somdeel2}
\end{align}
 The sum in \eqref{somdeel2} does not have multiplicative terms, so we cannot write the sum of $\mu\l(d\r)  \delta \l(u_0(d),d^2,g\r)$ as an Euler product. However, we will see that the sums in \eqref{somdeel1} and \eqref{somdeel2} have an Euler product as a common factor.
This helps us to write our sum of $\mu\l(d\r)  \delta \l(u_0(d),d^2,g\r)$ as the product of an Euler product and a sum containing non-multiplicative terms.
\subsection{The sum in \eqref{somdeel1}} \label{s:som1}
For $d \in \mathbb{N}$ squarefree with $(ab,d)=1$ and $(a+b,d,h)=1$, we apply Lemma \ref{technisch} to obtain
\begin{align} \label{gedeelddoorC(h)}
\frac{A(u_0(d),d^2,h)}{C(h)} = \prod \limits_{\substack{p|a+b \\ p|d}} \l(1- \frac{1}{p} \r)  \prod \limits_{\substack{p | d \\ p|h}} \l( \frac{1}{1- \frac{1}{p-1}} \r) \prod \limits_{\substack{p | d \\ p \nmid h}}  \l( \frac{1}{1- \frac{1}{p(p-1)}} \r).
\end{align}
We see that the map $ \psi:\mathbb{N} \to \mathbb{C}$ given by
\begin{align} \label{psi}
 \psi(d):=
\begin{cases}
\frac{A(u_0(d),d^2,h)}{C(h)}, & \text{if } (ab,d)=1 \text{ and } (a+b,d,h)=1,\\
0, & \text{else,}
\end{cases}
\end{align}
is multiplicative, hence we have an Euler product
\begin{align} \label{hugeeulerproduct}
\begin{split}
&\sum \limits_{\substack{d =1\\(ab,d)=1  \\ (a+b,d,h)=1}}^{\infty} \mu\l(d\r) \frac{A(u_0(d),d^2,h)}{\phi(d^2)} 
\\& =C(h) \prod \limits_{\substack{p \nmid ab \\ p|a+b \\ p \nmid h}} \l( 1 - \frac{p-1}{p^3-p^2-p}  \r)
\prod \limits_{\substack{p \nmid ab \\ p \nmid a+b \\ p | h}} \l( 1 - \frac{1}{p^2-2p}  \r) 
\prod \limits_{\substack{p \nmid ab \\ p \nmid a+b \\ p \nmid h}} \l( 1 - \frac{1}{p^2-p-1} \r).
\end{split}
\end{align}
\subsection{The sum in \eqref{somdeel2}} \label{s:som2}
We write all the $d$ in \eqref{somdeel2} as $d=d_0d_1$ where $d_0=(\Delta,d)$.
We show that the expressions $\l( \frac{\gamma(d)}{u_0(d)} \r)$ and $\mu(2|b_M(d)|)$, appearing in \eqref{somdeel2} only depend on $d_0$ and not on $d_1$. This will enable us to write the sum in \eqref{somdeel2} as a product of a sum over $d_0$ and an Euler product in $d_1$.
\begin{lemma} \label{gamma}
For $d \in \mathbb{N}$ squarefree and such that $(a,d)=1$, we have
$$\l( \frac{\gamma}{u_0(d)} \r) = \l( \frac{\gamma}{-ab} \r),$$
where $\gamma:= \gamma(d)$.
\end{lemma}
\begin{proof} Theorem 2.2.15 from \cite{cohen} states that if $D$ is a fundamental discriminant, then the Kronecker symbol $\l( \frac{D}{ \cdot} \r)$ is periodic modulo $|D|$. 
As stated in the proof of Lemma 2.4 from \cite{moree2008}, $\gamma$ is a fundamental discriminant, so $\l( \frac{\gamma}{\cdot} \r)$ is periodic mod $|\gamma|$. Now $\gamma$ divides $d^2$, while $au_0(d)+b \equiv 0 \mod d^2$, hence 
$au_0(d) \equiv - b \mod |\gamma|$, so by periodicity we get $$\l( \frac{\gamma}{a} \r) \l( \frac{\gamma}{u_0(d)} \r)= \l( \frac{\gamma}{au_0(d)} \r)= \l( \frac{\gamma}{-b} \r).$$ Because $(a,d)=1$ we have $(a,\gamma)=1$, hence
$\pm 1= \l( \frac{\gamma}{a} \r)= \l( \frac{\gamma}{a} \r)^{-1}$. We conclude that
$$ \l( \frac{\gamma}{u_0(d)} \r) = \frac{\l( \frac{\gamma}{-b} \r)}{\l( \frac{\gamma}{a} \r)}
= \l( \frac{\gamma}{-b} \r) \l( \frac{\gamma}{a} \r)= \l( \frac{\gamma}{-ab} \r).$$
\end{proof}

Notice that for $d$ as in Lemma \ref{gamma}, we have $(d^2 , \Delta) = (d_0^2, \Delta)$, hence
 $b_M(d)=  b_M(d_0)$. Combined with Lemma \ref{gamma} this shows that $\gamma(d)= \gamma(d_0)$ and $\l( \frac{\gamma(d)}{u_0(d)} \r) = \l( \frac{\gamma(d_0)}{-ab} \r)$.
%
Using these observations and the fact that the map $\psi$ defined in \eqref{psi} is multiplicative, we get
\begin{align} \label{sum2outcome}
 &\sum \limits_{\substack{d =1\\(ab,d)=1  \\ (a+b,d,h)=1}}^{\infty} \mu\l(d\r) \frac{A(u_0(d),d^2,h)}{\phi(d^2)} \l( \frac{\gamma(d)}{u_0(d)} \r) \frac{\mu(2|b_M(d)|)}{ \prod \limits_{\substack{p|b_M(d) \\ p|h}} (p-2) \prod \limits_{\substack{p|b_M(d) \\ p \nmid h}} (p^2-p-1)} =
 \\& C(h)\sum \limits_{\substack{d_0 | \Delta \\(ab,d_0)=1  \\ (a+b,d_0,h)=1}} \mu\l(d_0\r) \frac{A(u_0(d_0),d_0^2,h)}{C(h) \phi(d_0^2)} \l( \frac{\gamma(d_0)}{-ab} \r) \frac{\mu(2|b_M(d_0)|)}{ \prod \limits_{\substack{p|b_M(d_0) \\ p|h}} (p-2) \prod \limits_{\substack{p|b_M(d_0) \\ p \nmid h}} (p^2-p-1)} \times \label{sumoverd_0}
 \\& \times \sum \limits_{\substack{(d_1,\Delta)=1 \\(ab,d_1)=1  \\ (a+b,d_1,h)=1}} \mu\l(d_1\r) \frac{A(u_0(d_1),d_1^2,h)}{C(h) \phi(d_1^2)}. \label{sumoverd_1}
\end{align}
The sum in \eqref{sumoverd_1} is equal to the Euler product
\begin{align} \label{gemfactor}
\prod \limits_{\substack{p \nmid ab \\ p|a+b \\ p \nmid h \\ p \nmid \Delta}} \l( 1 - \frac{p-1}{p^3-p^2-p}  \r)
\prod \limits_{\substack{p \nmid ab \\ p \nmid a+b \\ p | h \\ p \nmid \Delta}} \l( 1 - \frac{1}{p^2-2p}  \r) 
\prod \limits_{\substack{p \nmid ab \\ p \nmid a+b \\ p \nmid h \\ p \nmid \Delta}} \l( 1 - \frac{1}{p^2-p-1} \r).
\end{align} 

\subsection{Putting it all together}

We substitute the equations \eqref{hugeeulerproduct}-\eqref{gemfactor} in the equation \eqref{somdeel1}-\eqref{somdeel2}. Notice that the terms in \eqref{somdeel1} and \eqref{somdeel2} have a common factor, namely $C(h)$ times the Euler product in \eqref{gemfactor}. We get that 
$\sum \limits_{\substack{d =1\\(a,d)=1  \\ (u_0(d),d)=1}}^{\infty} \mu\l(d\r)  \delta \l(u_0(d),d^2,g\r)$ equals
\begin{align*} 
\\& C(h)\prod \limits_{\substack{p \nmid ab \\ p|a+b \\ p \nmid h \\ p \nmid \Delta}} \l( 1 - \frac{p-1}{p^3-p^2-p}  \r)
\prod \limits_{\substack{p \nmid ab \\ p \nmid a+b \\ p | h \\ p \nmid \Delta}} \l( 1 - \frac{1}{p^2-2p}  \r) 
\prod \limits_{\substack{p \nmid ab \\ p \nmid a+b \\ p \nmid h \\ p \nmid \Delta}} \l( 1 - \frac{1}{p^2-p-1} \r) \times
\\& \times \Bigg(
 \prod \limits_{\substack{p \nmid ab \\ p|a+b \\ p \nmid h \\ p | \Delta}} \l( 1 - \frac{p-1}{p^3-p^2-p}  \r)
\prod \limits_{\substack{p \nmid ab \\ p \nmid a+b \\ p | h \\ p | \Delta}} \l( 1 - \frac{1}{p^2-2p}  \r) 
\prod \limits_{\substack{p \nmid ab \\ p \nmid a+b \\ p \nmid h \\ p | \Delta}} \l( 1 - \frac{1}{p^2-p-1} \r)
\\& +\sum \limits_{\substack{d_0 | \Delta \\(ab,d_0)=1  \\ (a+b,d_0,h)=1}} \mu\l(d_0\r) \frac{A(u_0(d_0),d_0^2,h)}{C(h) \phi(d_0^2)} \l( \frac{\gamma(d_0)}{-ab} \r) \frac{\mu(2|b_M(d_0)|)}{ \prod \limits_{\substack{p|b_M(d_0) \\ p|h}} (p-2) \prod \limits_{\substack{p|b_M(d_0) \\ p \nmid h}} (p^2-p-1)} 
\Bigg).
\end{align*}
This can be written as
\begin{align*} 
&\prod \limits_{\substack{p \nmid ab \\ p|a+b \\ p \nmid h \\ p \nmid \Delta}} \l( 1 - \frac{p-1}{p^3-p^2-p}  \r)
\prod \limits_{\substack{p \nmid ab \\ p \nmid a+b \\ p | h \\ p \nmid \Delta}} \l( 1 - \frac{1}{p^2-2p}  \r) 
\prod \limits_{\substack{p \nmid ab \\ p \nmid a+b \\ p \nmid h \\ p \nmid \Delta}} \l( 1 - \frac{1}{p^2-p-1} \r) \times
\\& 
\times \sum \limits_{\substack{d_0 | \Delta \\(ab,d_0)=1  \\ (a+b,d_0,h)=1}} \mu\l(d_0\r) \frac{A(u_0(d_0),d_0^2,h)}{\phi(d_0^2)} \l( 1+ \l( \frac{\gamma(d_0)}{-ab} \r) \frac{\mu(2|b_M(d_0)|)}{ \prod \limits_{\substack{p|b_M(d_0) \\ p|h}} (p-2) \prod \limits_{\substack{p|b_M(d_0) \\ p \nmid h}} (p^2-p-1)}  \r).
\end{align*}
Recalling Theorem \ref{thm:moree2008}, Definition \ref{deltanatural} and Lemma \ref{technisch} (ii), we can rewrite the above equation as 
\begin{align*} 
\sum \limits_{\substack{d =1\\(a,d)=1  \\ (u_0(d),d)=1}}^{\infty} \mu\l(d\r)  \delta \l(u_0(d),d^2,g\r) =
  \prod \limits_{ \substack{ p \nmid \Delta \\ p \nmid ab \\ p \nmid (a+b,h)  }} \l(1- \delta^{\natural}(u_0(p),p^2,g) \r) \sum \limits_{\substack{d_0 | \Delta \\ (a,d_0)=1}} \mu\l(d_0\r)  \delta (u_0(d_0),d_0^2,g).
\end{align*}
\begin{lemma}
We have
$$\sum \limits_{\substack{d_0 | \Delta \\ (a,d_0)=1}} \mu\l(d_0\r)  \delta (u_0(d_0),d_0^2,g)= \sum \limits_{\substack{ c \mod |\Delta| \\ \forall p: p| \Delta \implies p^2 \nmid ac+b}} \delta(c, |\Delta|,g).$$
\end{lemma}
\begin{proof}
Since $a$ and $b$ are coprime, primes dividing $ac+b$ cannot divide $a$. By the inclusion-exclusion principle we get
\begin{align*}
\sum \limits_{\substack{ c \mod |\Delta| \\ \forall p: p| \Delta \implies p^2 \nmid ac+b}} \delta(c, |\Delta|,g)&= \sum \limits_{\substack{ c \mod |\Delta| }} \delta(c, |\Delta|,g) - \sum \limits_{\substack{ c \mod |\Delta| \\ \exists p: p| \Delta , ~p^2 | ac+b, ~ p \nmid a}} \delta(c, |\Delta|,g)
\\&= \sum \limits_{\substack{d_0| \Delta \\ (a,d_0)=1 }} \mu(d_0) \sum \limits_{\substack{c \mod |\Delta| \\ d_0^2|ac+b}} \delta(c, |\Delta|,g).
\end{align*}
Assume $(a,d_0)=1$ and $d_0| \Delta$. Then we have $d_0^2|(ac+b)$ if and only if \\ $c \equiv u_0(d_0) \mod d_0^2$, by Notation \ref{u_0(d)}, so in this case we have
$$\sum \limits_{\substack{c \mod |\Delta| \\ d_0^2|ac+b}} \delta(c, |\Delta|,g)=
\sum \limits_{\substack{c \mod |\Delta| \\ c \equiv u_0(d_0) \mod d_0^2}} \delta(c, |\Delta|,g) = \delta(u_0(d_0),d_0^2,g).$$
\end{proof}
From this lemma we obtain
\begin{align*}
\sum \limits_{\substack{d =1\\(a,d)=1  \\ (u_0(d),d)=1}}^{\infty} \mu\l(d\r)  \delta \l(u_0(d),d^2,g\r)= \prod \limits_{ \substack{ p \nmid \Delta \\ p \nmid ab \\ p \nmid (a+b,h)  }} \l(1- \delta^{\natural}(u_0(p),p^2,g) \r)  \sum \limits_{\substack{ c \mod |\Delta| \\ \forall p: p| \Delta \implies p^2 \nmid ac+b}} \delta(c, |\Delta|,g).
\end{align*} 
Substituting this in \eqref{endoffirstpart} completes the proof of Theorem \ref{ap+b}.
Also we notice that
\begin{align*}
 &\prod \limits_{ \substack{ p \nmid \Delta \\ p \nmid ab \\ p \nmid (a+b,h)  }} \l(1- \delta^{\natural}(u_0(p),p^2,g) \r)=\\&  \prod \limits_{\substack{p \nmid ab \\ p|a+b \\ p \nmid h \\ p \nmid \Delta}} \l( 1 - \frac{p-1}{p^3-p^2-p}  \r)
\prod \limits_{\substack{p \nmid ab \\ p \nmid a+b \\ p | h \\ p \nmid \Delta}} \l( 1 - \frac{1}{p^2-2p}  \r) 
\prod \limits_{\substack{p \nmid ab \\ p \nmid a+b \\ p \nmid h \\ p \nmid \Delta}} \l( 1 - \frac{1}{p^2-p-1} \r)>0,
\end{align*}
which we needed for the proof of Corollary \ref{cor:ap+b}.
\begin{acknowledgements}
This paper is a continuation of the author's bachelor thesis, which was supervised by E. Sofos at Leiden University. The author would like to thank him for his advice on the mathematics and the writing of both the thesis and this paper.
\end{acknowledgements}

\bibliographystyle{IEEEtrans} 
\bibliography{Refmainterm}

\end{document}